\documentclass[a4paper,11pt]{article}
\usepackage{graphicx}
\usepackage{amssymb}
\usepackage[french, english]{babel}
\usepackage[utf8]{inputenc}
\usepackage{amsmath,amsfonts,amssymb,amsthm}
\usepackage[T1]{fontenc}
\usepackage{fullpage}
\usepackage{hyperref}
\usepackage{mathrsfs}
\usepackage{relsize}
\usepackage{tikz}
\usepackage{bbm}
\usepackage{appendix}
\usepackage{pifont}
\usetikzlibrary{patterns}

\definecolor{darkgreen}{rgb}{0.0, 0.5, 0.0}

\newcommand{\eps}{\varepsilon}

\newcommand{\ct}{\text{ct}}

\newcommand{\C}{\mathcal C}
\newcommand{\Ce}{C_{\eps_\theta}}

\newcommand{\E}{\mathbb E}

\renewcommand{\P}{\mathbb P}
\newcommand{\cP}{\mathcal P}

\newcommand{\U}{\mathcal{U}}

\newcommand{\1}{\mathds 1}

\newcommand{\cQ}{\mathcal{Q}}

\newcommand{\diam}{\mathrm{diam}}

\usepackage{physics}

\theoremstyle{definition}

\newtheorem{thm}{Theorem}

\newtheorem{defn}{Definition}

\newtheorem{prop}[defn]{Proposition}

\newtheorem{lem}[defn]{Lemma}
\newtheorem{conj}[defn]{Conjecture}

\tikzstyle{every node}=[circle, draw, fill=black!50, inner sep=0pt, minimum width=4pt]
\tikzstyle{white}=[circle, draw, fill=black!0, inner sep=0pt, minimum width=4pt]
\tikzstyle{bigwhite}=[circle, draw, fill=black!0, inner sep=0pt, minimum width=10pt]
\tikzstyle{dual}=[circle, draw=blue, fill=black!0, inner sep=0pt, minimum width=4pt]
\tikzstyle{fat}=[circle, draw, fill=red!50, inner sep=0pt, minimum width=8pt]
\tikzstyle{fat_bis}=[circle, draw, fill=blue!50, inner sep=0pt, minimum width=8pt]
\tikzstyle{fat_ter}=[circle, draw, fill=green!50, inner sep=0pt, minimum width=8pt]
\tikzstyle{rouge}=[circle, draw, fill=red, inner sep=0pt, minimum width=7pt]
\tikzstyle{bleu}=[circle, draw, fill=blue, inner sep=0pt, minimum width=7pt]
\tikzstyle{petitrouge}=[circle, draw, fill=red, inner sep=0pt, minimum width=4pt]
\tikzstyle{petitbleu}=[circle, draw, fill=blue, inner sep=0pt, minimum width=4pt]
\tikzstyle{texte}=[draw=none, fill=none]

\title{\bf{Planarity and non-separating cycles in uniform high genus quadrangulations}}
\author{Baptiste \bsc{Louf}\footnote{Uppsala universitet, \url{baptiste.louf@math.uu.se}}}

\begin{document}

\maketitle

\paragraph{Abstract.} We study large uniform random quadrangulations whose genus grow linearly with the number of faces, whose local convergence was recently established by Budzinski and the author \cite{BL19,BL20}.
Here we  study several properties of these objects which are not captured by the local topology. Namely we show that balls around the root are planar whp up to logarithmic radius, and we prove that there exists short non-contractible cycles with positive probability.
\section{Introduction}

\paragraph{Planar maps} Maps are surfaces formed by gluing polygons together. They have been given a lot of attention in the last decades, especially in the case of planar maps, i.e. maps of the sphere. They were first approached from the combinatorial point of view, starting with their exact enumeration by Tutte \cite{Tut63}, with generating function methods. Later on,  bijections between maps and decorated trees were discovered, starting with the Cori--Vauquelin--Schaeffer bijection \cite{Sch98these}.

More recently, thanks to both enumerative and bijective results, the properties of large random maps have been studied. More precisely, one can study the geometry of random maps picked uniformly in certain classes, as their size tends to infinity. In the case of planar maps, the most notable results are probably the identification of two types of "limits" (for two well defined topologies on the set of planar maps): the local limit (the \emph{UIPT}\footnote{in the case of triangulations, i.e. maps made out of triangles.} \cite{AS03} and \emph{UIPQ}\footnote{in the case of quadrangulations.} \cite{Kri05,CD06,Men08}) and the scaling limit (the \emph{Brownian map} \cite{LG11,Mie11}).

\paragraph{Maps on other surfaces}

Similar results exist for maps of genus greater than $0$. It is possible to study uniform maps with a fixed genus $g>0$. Enumerative (asymptotic) results have been obtained (see for instance \cite{BC86}), and there are bijections for maps on any surface (see for instance \cite{CMS09}). On the probabilistic side, equivalents of the Brownian map in genus $g>0$ have been constructed \cite{Bet16}.

It is also possible to study maps without constraints on the genus, see \cite{CP16,BCDH18,BCP19} for three different approaches to this problem.

\paragraph{High genus maps}

Very recently, yet another regime has been considered: maps whose genus grows linearly in the size of the map. They exhibit hyperbolic features, as their average degree (which is directly linked to the average curvature of the map) is asymptotically higher than in planar (or fixed genus) maps. Some of their geometric properties have been studied, starting with uniform unicellular maps, i.e. maps with only one face. Their local limit is a supercritical Galton-Watson tree \cite{ACCR13}, and their diameter is logarithmic \cite{Ray13a}. These two results rely on a bijection between unicellular maps and decorated trees \cite{CFF13}. 

The general case (i.e. uniform maps with various kinds of constraints on the face degrees) has been studied more recently, starting with uniform high genus triangulations, which converge locally in distribution towards a random hyperbolic triangulation of the plane \cite{BL19}. A larger family of maps (bipartite maps with prescribed face degrees) is studied in \cite{BL20}, and a similar behaviour is observed.


In this paper, we investigate global geometric features of high genus maps. For technical reasons, we will study quadrangulations instead of triangulations. More precisely, for the rest of the paper, fix $0<\theta<\frac{1}{2}$, and let $(g_n)$ be a sequence such that $\frac{g_n}{n}\to \theta$. Let $q^{(n)}$ be a uniform bipartite quadrangulation of genus $g_n$ with $2n$ faces. 

\paragraph{Curves on a surface}
In this work, we will give a lot of attention to cycles seen as curves on a surface. On a surface of  genus at least $2$, there are three different types of (simple, closed) curves (see Figure~\ref{fig_type_curves}). The first kind is \emph{contractible} curves, i.e. curves that can be continuously deformed into a point. There are two types of non-contractible curves: the \emph{separating} curves and the \emph{non-separating} curves. More precisely, given a connected surface $\mathcal{S}$ and a non-contractible curve $\mathcal C$ on $\mathcal{S}$, then we say that $\mathcal C$ is separating if and only if $\mathcal{S}\setminus \mathcal C$ is disconnected.

\begin{figure}[!h]
\center
\includegraphics[scale=0.6]{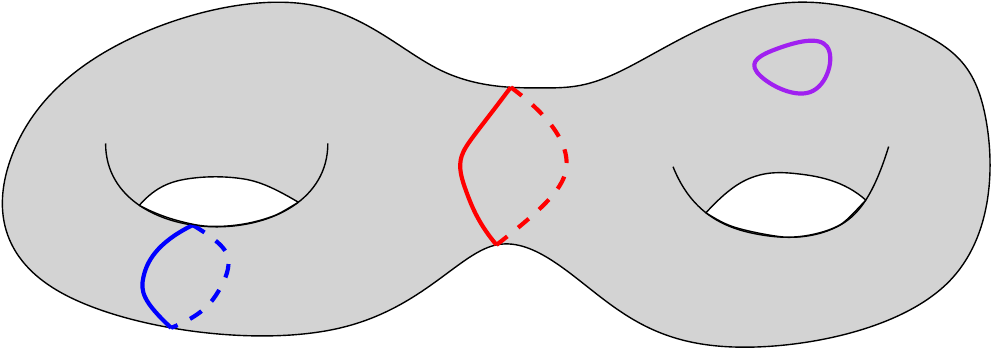}
\caption{The three types of curves on a surface: contractible (in purple), non-contractible and separating (in red) and non-separating (in blue).}\label{fig_type_curves}
\end{figure}

\paragraph{Our result}
In \cite{BL20}, it is proven that the local limit of $q^{(n)}$ is a random infinite quadrangulation of the plane. In particular, it implies that for every fixed $r$, the ball of radius $r$ around the root is planar with probability $1-o(1)$ as $n\to\infty$. This result might seem counter-intuitive at first, as $q^{(n)}$ is highly non-planar. Actually, here we extend this result to balls of a much larger radius. More precisely, we define the \emph{planarity radius} to be the largest $r$ such that the ball of radius $r$ around the root is planar \textbf{and} does not contain any non-contractible cycle. Actually, the second condition implies the first one, but they are not equivalent, see Figure~\ref{fig_cylinder}. 

\begin{figure}[!h]
\center
\includegraphics[scale=0.6]{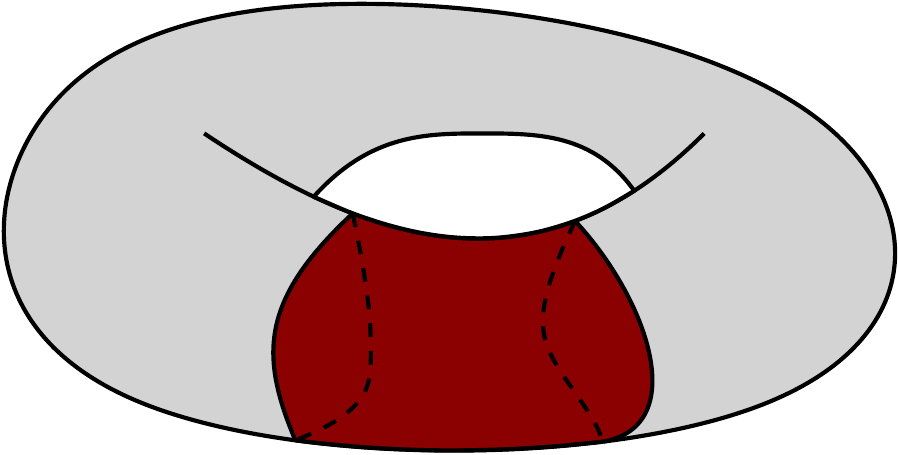}
\caption{The red part contains a non contractible curve of the grey surface, but it has the topology of a cylinder, hence it is planar.}\label{fig_cylinder}
\end{figure}
\vspace{1cm}
We show here that the planarity radius of $q^{(n)}$ is of logarithmic order whp\footnote{throughout the paper, we will write \emph{whp} instead of "with probability $1-o(1)$ as $n\to\infty$".}.

\begin{thm}\label{thm_planar_neighborhood}
There exists a constant $a_\theta$ such that
\[\text{PR}(q^{(n)})\geq a_\theta \log n\]
with probability $1-o(1)$ as $n\to \infty$, where $\text{PR}$ is the planarity radius.
\end{thm}
In terms of the dimensions of $q^{(n)}$, this is very big, as its diameter is also believed to be of logarithmic order (see Conjecture~\ref{conj_diameter}).

The neighbourhood of the root is planar, and there is no short non-contractible cycle passing near the root. However, if we look at the whole map, then there is a good chance to find very short non-contractible cycles.

\begin{thm}\label{thm_short_cycle}
There exists a constant $k_\theta>0$ such that
\[\P\left(\mbox{there exists a non-separating cycle of length $2$ in $q^{(n)}$}\right)\geq k_\theta +o(1)\]
as $n\to \infty$.
\end{thm}

The proofs of these theorems rely on asymptotic estimations of the number of high genus maps. The three main ingredients are the Carrell--Chapuy formula \cite{CC15}, the bounded ratio lemma of  \cite{BL20}, and the following result:

\begin{prop}\label{prop_injection}
For all $n\geq 1,g\geq 1$, we have
\[2g Q(n,g)\leq (2n)^3Q(n,g-1)\]
where $Q(n,g)$ is the number of bipartite quadrangulations of genus $g$ with $n$ faces.
\end{prop}

This proposition is proven by a combinatorial injective operation. 

\paragraph{Structure of the paper} We start with some definitions, then we discuss some natural developments of this work.
In the third section we prove Proposition~\ref{prop_injection}, and the last two sections are devoted to the proofs of Theorems~\ref{thm_planar_neighborhood} and~\ref{thm_short_cycle}. 

\paragraph{Acknowledgements} The author is grateful to Thomas Budzinski and Guillaume Chapuy for useful comments about this work, to Bram Petri for explanations about hyperbolic geometry, and to the anonymous referee whose comments allowed to greatly improve the presentation of this paper.

\section{Definitions}\label{sec_def}

A \emph{map} $M$ is the data of the gluing of a finite collection of oriented polygons (the \emph{faces}) to form a compact, connected and oriented surface. The vertices and sides of the polygons become, after the gluing, the \emph{vertices} and \emph{edges} of $M$. The genus $g$ of the surface formed by the gluing of the polygons is also called the genus of $M$. A bipartite map is a map whose vertices are either black or white, and whose edges always connect a black and a white vertex. We will consider \emph{rooted} bipartite maps, i.e. bipartite maps with a distinguished edge, called the \emph{root}. The white vertex incident to the root is called the \emph{root vertex}. We can give a canonical orientation to all edges (from white to black for instance), and therefore it makes sense to talk about what is on the left or on the right of an edge.

A \emph{map with holes} is a bipartite map with a certain number of marked faces called holes. For any map $m$, we define the \emph{ball of radius $r$} around the root of $m$ (noted $B_r(m)$) as the map with holes formed by all vertices of $m$ at distance $r$ or less to the root vertex, all edges of $m$ with both endpoints in $B_r(m)$, and every face of $m$ that has all its incident edges in $B_r(m)$. 

A (bipartite) \emph{quadrangulation}\footnote{from now on, all quadrangulations will be bipartite.} $q$ is a bipartite map whose faces are quadrangles. If $q$ has $n$ faces and genus $g$, it has $2n$ edges and $n+2-2g$ vertices. We denote by $\cQ(n,g)$ the set of triangulations with $n$ faces and genus $g$, and by $Q(n,g)$ its cardinality. A quadrangulation with a \emph{boundary} of size $2p$ is a map with quadrangular faces except for one special face, called the boundary, that is a simple\footnote{by simple, we mean that in the construction of the map by gluing polygons, no two sides of the $2p$-gon are glued together.} $2p$-gon, such that the boundary sits on the right of the root. We denote by $\cQ^{(p)}(n,g)$ the set of bipartite quadrangulation  with $n$ quadrangles, a boundary of size $p$ and genus $g$, and by $Q^{(p)}(n,g)$ it cardinal. Quadrangulations with two boundaries are defined the same way.
A quadrangulation with two boundaries is a map \textbf{with two roots} with quadrangular faces except for two special faces, called the boundaries, that are simple and vertex-disjoint, such that each boundary sits on the right of one root. We require the boundaries to be distinguishable, i.e. there is a first boundary and a second boundary.
We denote by $\cQ^{(p,p')}(n,g)$ the set of bipartite quadrangulations  with $n$ quadrangles, boundaries of size $2p$ and $2p'$, and genus $g$, and by $Q^{(p,p')}(n,g)$ its cardinality.

A \emph{unicellular map} is a map with only one face, with a distinguished oriented edge called the root. Let $\U(n,g)$ be the set of unicellular maps of genus $g$ with $n$ edges.

A \emph{simple path} of a map $M$ is a list of vertices $(v_0,v_1,…,v_\ell)$ and edges $(e_1,e_2,…,e_\ell)$, such that for all $1\leq 1\leq\ell$, $e_i$ joins $v_{i-1}$ and $v_i$, with the condition that the $v_i$'s are all distinct. The size of a simple path $P$, noted $|P|$, is the number of edges it contains.
A  \emph{cycle} of a map $M$ is a list of vertices $(v_0,v_1,…,v_\ell)$ and edges $(e_1,e_2,…,e_\ell)$, such that for all $1\leq 1\leq\ell$, $e_i$ joins $v_{i-1}$ and $v_i$, and such that $v_0=v_\ell$. The size of a cycle $C$, noted $|C|$, is the number of edges it contains. In what follows, we will only consider \emph{simple} cycles, that satisfy the extra condition $v_i\neq v_j$ for all $1\leq i<j\leq \ell$. A \emph{contractible cycle} is a cycle that, seen as a curve on the surface, is contractible.  
A  contractible simple cycle separates the map $M$ in two parts, one of them being a planar map with a boundary. On the other hand, a non-contractible cycle either separates the map in two non-planar parts, in which case it is called \emph{separating}, or does not separate the map, in which case it is called \emph{non-separating}. Note that a non-separating cycle is necessarily non contractible. Recall that Figure~\ref{fig_type_curves} presents the three types of curves on a surface.



\section{Discussion and conjectures}

Before going to the proofs of the main results, we want to compare our model with pre-existing models of hyperbolic geometry, and present a few problems that would be a natural extension of this work.

\paragraph{Comparison with hyperbolic geometry}

High genus maps can be seen as discrete models of two dimensional hyperbolic geometry. In the continuous setting, several models of random hyperbolic metrics on surfaces as the genus goes to infinity have been well studied in the past, two famous examples being the Brooks--Makover model \cite{BM04} and the Weil--Petersson measure \cite{GPY11,Mir13}. The results obtained so far about random uniform high genus maps are equivalent to the results obtained on these continuous models, and we conjecture that high genus maps will behave similarly as continuous models when we look at other geometric observables (see the conjectures below).

In particular, concerning the present work, in \cite{Mir13}, it is proved that the injectivity radius around a given point in a surface of genus $g$ with a hyperbolic metric under the Weil--Petersson measure grows logarithmically in $g$ as $g\to\infty$, and this implies the same growth rate for the planarity radius of such surfaces. The injectivity radius is defined as the smallest $r$ such that the ball of radius $r$ around a given point is not homeomorphic to a disk. Unfortunately, such a result about the injectivity radius would not transfer to maps, as they are not "smooth" enough.

\paragraph{Other models of maps}
We are quite confident that the proofs in the present article adapt to many other models of maps (at least triangulations and bipartite maps with prescribed bounded face degrees). The proof of Proposition~\ref{prop_injection} would involve objects called \emph{mobiles} (see \cite{BDG04}) in lieu of unicellular well-labeled maps, a version of the bounded ratio lemma holds for many models (see \cite{BL19,BL20}), and the Carrell--Chapuy formula for bipartite quadrangulations can be replaced by other similar formulas \cite{GJ08,Louf19b}.

\paragraph{The diameter of high genus maps}
In Theorem~\ref{thm_planar_neighborhood}, we only give a lower bound for the radius of the biggest planar ball around the root. We believe that the upper bound is also of logarithmic order, and this would be implied by the following conjecture\footnote{ we want to stress on the fact that this conjecture is not ours, instead it is attributable to several people in the community.}:
\begin{conj}\label{conj_diameter}
There exist constants $m_\theta$ and $M_\theta$ such that
\[m_\theta\log(n)\leq \diam(q^{(n)})\leq M_\theta\log(n)\]
whp.
\end{conj}
The lower bound is an immediate corollary of Theorem~\ref{thm_planar_neighborhood}, but we must mention that there exists a simpler proof of the lower bound (G. Chapuy, private communication).

\paragraph{Convergence of the ratio}
If we take $\theta\in(0,1/2)$, Proposition~\ref{prop_injection} along with the Carrell--Chapuy formula \eqref{eq_CC} imply that
\[\frac{Q(n,g_n)}{n^2Q(n,g_n-1)}=\Theta(1)\] as $n\to \infty$. We believe that this result can be made more precise.
\begin{conj}\label{conj_ratio}
There exists a function $r(\theta)$ such that 
\[\frac{Q(n,g_n)}{n^2Q(n,g_n-1)}\to r(\theta)\] as $n\to \infty$.
\end{conj}

\paragraph{Short non contractible cycles} The \emph{systole} of a map is the size of its shortest non-contractible cycle. Theorem~\ref{thm_short_cycle} proves that $\text{syst}(q^{(n)})=2$ with positive probability as $n\to\infty$. This leads us to think that the systole is asymptotically almost surely finite, and this would be coherent with results on continuous models \cite{Pet17}.
\begin{conj} \label{conj_syst} We have

\[\lim_{M\to \infty}\limsup_{n\to\infty}\P(\text{syst}(q^{(n)})>M)=0.\]

\end{conj}

We conjecture that the shortest non-contractible cycle is non-separating, while the shortest separating non-contractible cycle is actually much bigger.

\begin{conj}\label{conj_sepNC}
For any map $m$, let $SNC(m)$ be the size of the shortest separating non-contractible cycle of $m$. There exists a constant $s_\theta$ such that
\[SNC(q^{(n)})\geq s_\theta \log n\]
whp.
\end{conj}

The motivation for this conjecture might seem a bit more obscure than the previous ones, it actually comes from (conjectural) asymptotic  estimation of the terms of the Carrell--Chapuy formula \eqref{eq_CC}, and again a similar result exists on continuous models \cite{Mir13,PWX20}.

\section{Proof of Proposition~\ref{prop_injection}}
Proposition~\ref{prop_injection} is a quite direct consequence of the bijections of \cite{CMS09} and \cite{Ch11}. We will only briefly recall the details of these two bijections.

A unicellular map is said to be \emph{well-labelled} (see Figure~\ref{fig_well_labeled}) if each of its vertices carries an integer label, and the two following conditions are verified:
\begin{itemize}
\item the minimal label is $1$,
\item the labels of two adjacent vertices differ by at most $1$.
\end{itemize}
Let $\U^{lab}(n,g)$ be the set of well-labeled unicellular maps of genus $g$ with $n$ edges, and $U^{lab}(n,g)$ its cardinal.

We know that there is a $2$-to-$1$ correspondence \cite{CMS09} between $\U^{lab}(n,g)$ and the set of maps of $\cQ(n,g)$ with a distinguished vertex, therefore
\begin{equation}\label{eq_well-lab_quadr}
2U^{lab}(n,g)=(n+2-2g)Q(n,g).
\end{equation}

\begin{figure}
\center
\includegraphics[scale=0.5]{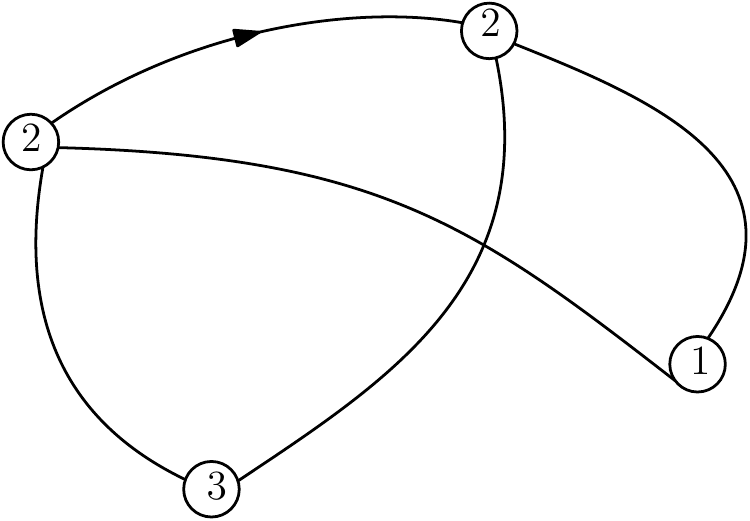}
\caption{A well labeled unicellular map.}\label{fig_well_labeled}
\end{figure}

A \emph{trisection} of a unicellular map is a special corner of this map defined in \cite{Ch11}. We do not give the precise definition of a trisection here, as it is not needed, but we underline two key properties of trisections:
\begin{itemize}
\item if $c$ is a trisection of a unicellular map $U\in\U(n,g)$, let $v$ be the vertex incident to $c$. Then there exist two other corners $c_1$ and $c_2$ incident to $v$ such that $v$ can be split along $c$, $c_1$ and $c_2$ as in Figure~\ref{fig_split_trisection} and that the resulting map belongs to $\U(n,g-1)$,
\item there are $2g$ trisections in a map of $\U(n,g)$.
\end{itemize}

The first property is explained in Section 2.3 together with Definition 2 of \cite{Ch11}. More precisely, in Section 2.3, it is explained that it is possible to perform a "slicing operation" (which we call splitting here) around three "intertwined half edges", and Definition 2 explains how a trisection involves three intertwined half edges. The second property is Lemma 3 in \cite{Ch11}.

\begin{figure}
\center
\includegraphics[scale=0.8]{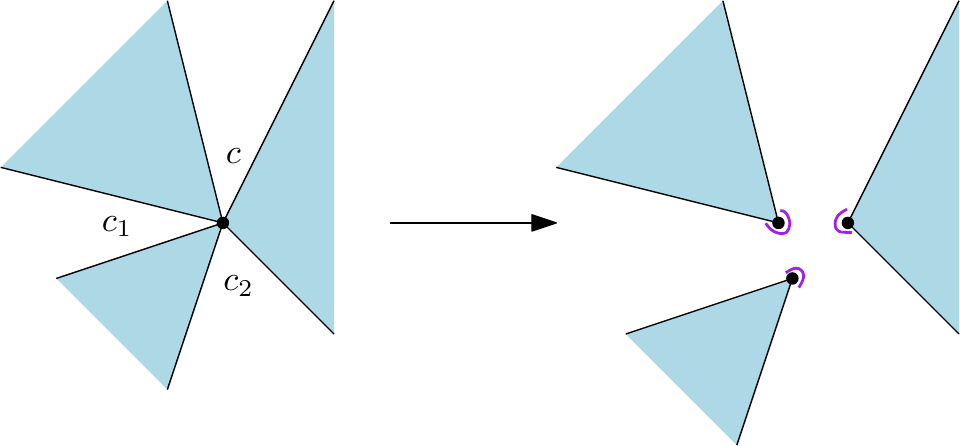}
\caption{Splitting a trisection.}\label{fig_split_trisection}
\end{figure}
The first point provides an injective operation from the set of maps of $\U(n,g)$ with a marked trisection to the set of maps of $\U(n,g-1)$ with three marked corners. Note that this injection adapts to well-labelled unicellular maps, if we decide that the vertex incident to the marked trisection is split into three vertices with the same label.

Now, by the second point, and since there are $2n$ corners in a map with $n$ edges, we have

\begin{equation}\label{eq_injection_trisec}
2gU^{lab}(n,g)\leq (2n)^3U^{lab}(n,g-1).
\end{equation}

If we combine equations~\eqref{eq_well-lab_quadr} and~\eqref{eq_injection_trisec}, we prove Proposition~\ref{prop_injection}.


\section{Planar neighbourhoods of the root}
In this section, we will prove Theorem~\ref{thm_planar_neighborhood}. We will introduce objects called cycles with tails, that are close to cycles passing through the root. Proposition~\ref{prop_cycle_tail} gives a lower bound on the size of these objects in high genus bipartite quadrangulations, and we will use it to prove Theorem~\ref{thm_planar_neighborhood}.
\subsection{Cycles with tails}
Theorem~\ref{thm_planar_neighborhood} is actually a corollary of the following proposition regarding non-contractible cycles that pass through the root. More precisely, we define a \emph{cycle with tail} to be either a simple non-contractible cycle containing the root edge, or a simple non-contractible cycle attached to a simple path such that the root edge is on one end of the path, and that the vertex on the other end of the path belongs to the cycle, but the cycle and the path do not intersect anywhere else. The size of a cycle with tail $(P,C)$, noted $|(P,C)|$, is the number of edges it contains, i.e. $|(P,C)|=|P|+|C|$.

\begin{prop}\label{prop_cycle_tail}
Let $\ct(q^{(n)})$ be the size of the smallest cycle with tail in $q^{(n)}$. Then there exists a constant $c_\theta$ such that
\[\ct(q^{(n)})\geq c_\theta \log n \]
whp.
\end{prop}

Before proving this proposition, we first prove that it implies Theorem~\ref{thm_planar_neighborhood}.

\begin{proof}[Proof of Theorem~\ref{thm_planar_neighborhood}]
Let $r$ such that $B_r(q^{(n)})$ contains a non-contractible cycle $\C$ of $q^{(n)}$ (again, $B_r(q^{(n)})$ might be planar, see Figure~\ref{fig_cylinder}). In what follows, we will only focus on the map $B_r(q^{(n)})$ and not $q^{(n)}$, all the lengths and distances are to be understood inside $B_r(q^{(n)})$. Assume $\C$ is of minimal length. If the root edge of $q^{(n)}$ does not belong to $\C$, let $(u,w)$ be its endpoints, and take a vertex $v\in \C$ such that
\[\min(d(v,u),d(v,w))\]
is minimal. Wlog, say that $d(v,u)=\min(d(v,u),d(v,w))$, and let $\cP$ be a shortest path from $w$ to $v$ starting with the root edge ($\cP$ exists since $v$ is closer to $u$ than it is to $w$). This ensures that $|\cP|\leq r+1$ and that $\cP$ and $\C$ intersect only at $v$. Hence, $(\cP,\C)$ is a cycle with tail, and since $B_r(q^{(n)})\subset q^{(n)}$, we have $|(\cP,\C)|\geq \ct(q^{(n)})$. Using Proposition~\ref{prop_cycle_tail}, we can conclude that
\begin{equation}\label{eq_borne_plus_petit_cycle}
|\C|\geq c_\theta \log n - r-1
\end{equation}
whp.

Now, let $v_1,v_2$ be a pair of vertices on $\C$ such that $d(v_1,v_2)$ is maximal. We can decompose  $\C$ into two paths $\cP'$ and $\cP''$ joining $v_1$ and $v_2$. Then, either $|\cP'|=d(v_1,v_2)$ or $|\cP''|=d(v_1,v_2)$. Indeed, if there exists a path $\cP^*$ between $v_1$ and $v_2$ satisfying $|\cP^*|<\min(|\cP'|,|\cP''|)$, then either $\cP'\cup\cP^*$ or $\cP''\cup\cP^*$ is a non-contractible cycle of length strictly shorter than $\C$, a contradiction (see Figure~\ref{fig_smallest_cycle} left). Wlog, say that $|\cP''|=d(v_1,v_2)$
 
Now, let $v_3$ be a vertex on $\cP'$, and say it separates $\cP'$ into $\cP_1$ (containing $v_1$) and $\cP_2$ (containing $v_2$). Then again, either $\cP_1$ or $\cP_2\cup\cP''$ is a path of minimal length between $v_1$ and $v_3$. But it cannot be $\cP_2\cup\cP''$, because it is strictly longer than $d(v_1,v_2)$, which was supposed to be maximal. Therefore, $|\cP_1|$ is less than or equal to $d(v_1,v_2)$, and the same goes for $|\cP_2|$ (see Figure~\ref{fig_smallest_cycle} right). Since we are in $B_r(q^{(n)})$, we also have $d(v_1,v_2)\leq 2r$, therefore
 
\begin{equation}\label{eq_cycle_rayon}
|\C|\leq |\cP''|+|\cP_1|+|\cP_2|\leq 3d(v_1,v_2) \leq 6r.
\end{equation}

Combining \eqref{eq_borne_plus_petit_cycle} and \eqref{eq_cycle_rayon}, we obtain
\[7r+1\geq c_\theta \log n\]
whp, which proves Theorem~\ref{thm_planar_neighborhood} for $a_\theta=\frac{c_\theta}{7}$.
\end{proof}

\begin{figure}
\center
\includegraphics[scale=0.5]{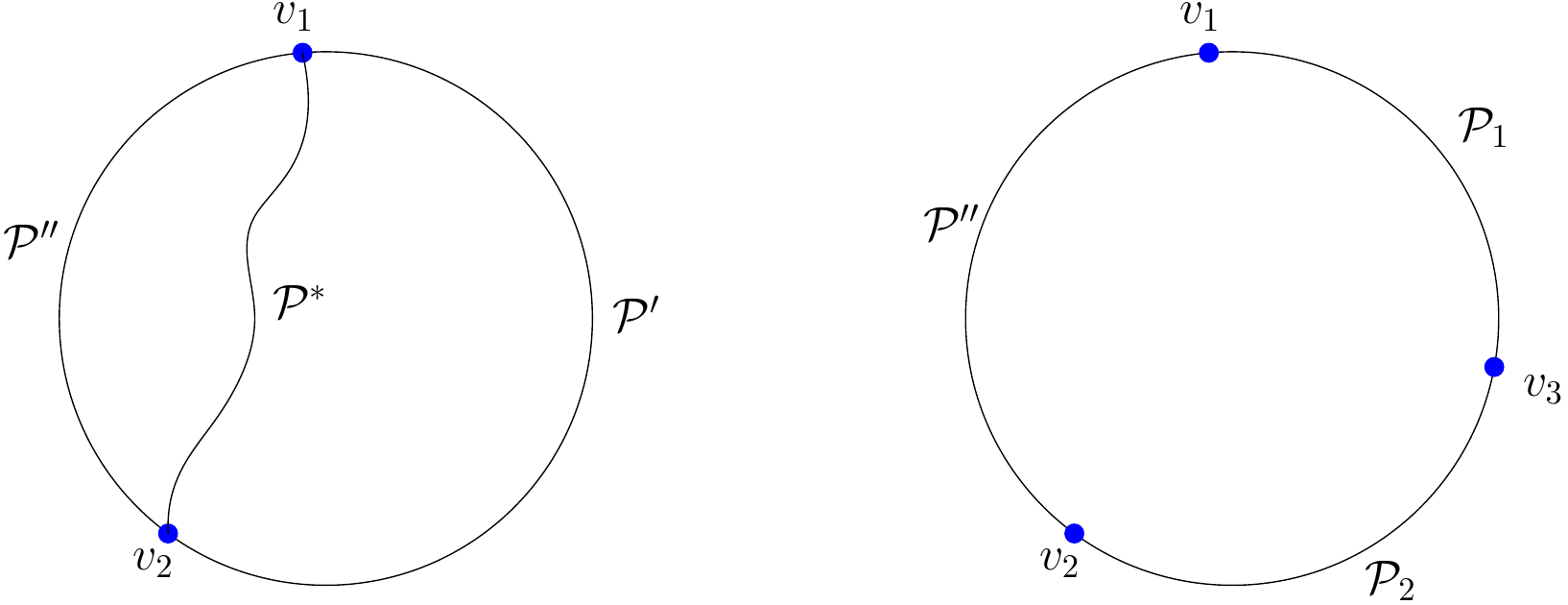}
\caption{Left: the cycle $\C$ contains a shortest path between $v_1$ and $v_2$. Right: any vertex on $\cP'$ splits it into to shortest paths towards $v_1$ and $v_2$.}\label{fig_smallest_cycle}
\end{figure}
\subsection{Two useful results}
Here we present two results  proved in previous works that will be useful for the proof of Proposition~\ref{prop_cycle_tail}. First we have the Carrell--Chapuy formula \cite{CC15}, which is a recurrence formula for enumerating bipartite quadrangulations in any genus. For every $n\geq 1,g\geq 0$, we have

\begin{equation}\label{eq_CC}
\begin{split}
(n+1)Q(n,g)=&4(2n-1)Q(n-1,g)+(2n-2)(n-1)(2n-1)Q(n-2,g-1)\\&+3\sum_{\substack{n_1+n_2=n-2\\n_1,n_2\geq 0}}\sum_{\substack{g_1+g_2=g\\g_1,g_2\geq 0}}(2n_1+1)Q(n_1,g_1)(2n_2+1)Q(n_2,g_2)
\end{split}
\end{equation}
with initial condition $Q(0,g)=\mathbbm{1}_{g=0}$.

Then we have the bounded ratio lemma  \cite[Lemma 13]{BL20}, that controls a certain growth rate for high genus maps:

\begin{lem}[The bounded ratio lemma]\label{lem_BRL}
For all $\eps>0$, there exists a constant $C_\eps>0$ such that for all $g\geq 0,n\geq 1$ satisfying $\frac{g}{n}\leq \frac{1}{2}-\eps$, we have
\[\frac{Q(n-1,g)}{Q(n,g)}\geq C_\eps.\]
\end{lem}

\subsection{Technical lemmas}
This section regroups a few technical lemmas that we will need to prove Proposition~\ref{prop_cycle_tail}. More precisely, we will use Lemmas~\ref{lem_ineq_boundaries} and~\ref{lem_ineq_genus}, as well as \eqref{eq_ineq_product}. The other lemmas of this section will only be used to establish \eqref{eq_ineq_product}.

We start with a bound on maps with boundaries.

\begin{lem}\label{lem_ineq_boundaries}
We have the following inequalities for all $n\geq 1, g\geq 0,p\geq 1, p'\geq 1$
\[Q^{(p)}(n,g)\leq Q(n+p-1,g)\quad \mbox{and} \quad Q^{(p,p')}(n,g)\leq 2(n+p+p'-2)Q(n+p+p'-2,g).\]
\end{lem}

\begin{proof}
We will prove the inequalities by an injective operation. Start with a map of $\cQ^{(p)}(n,g)$. If $p=1$, contract the boundary into the root edge. Otherwise, tessellate the boundary with $p-1$ quadrangles as in Figure~\ref{fig_tesselating} to obtain a map of $\cQ(n+p-1,g)$.

The proof of the second inequality is very similar, except that the second root becomes a marked edge, hence the factor $2(n+p+p'-2)$.
\begin{figure}
\center
\includegraphics[scale=0.5]{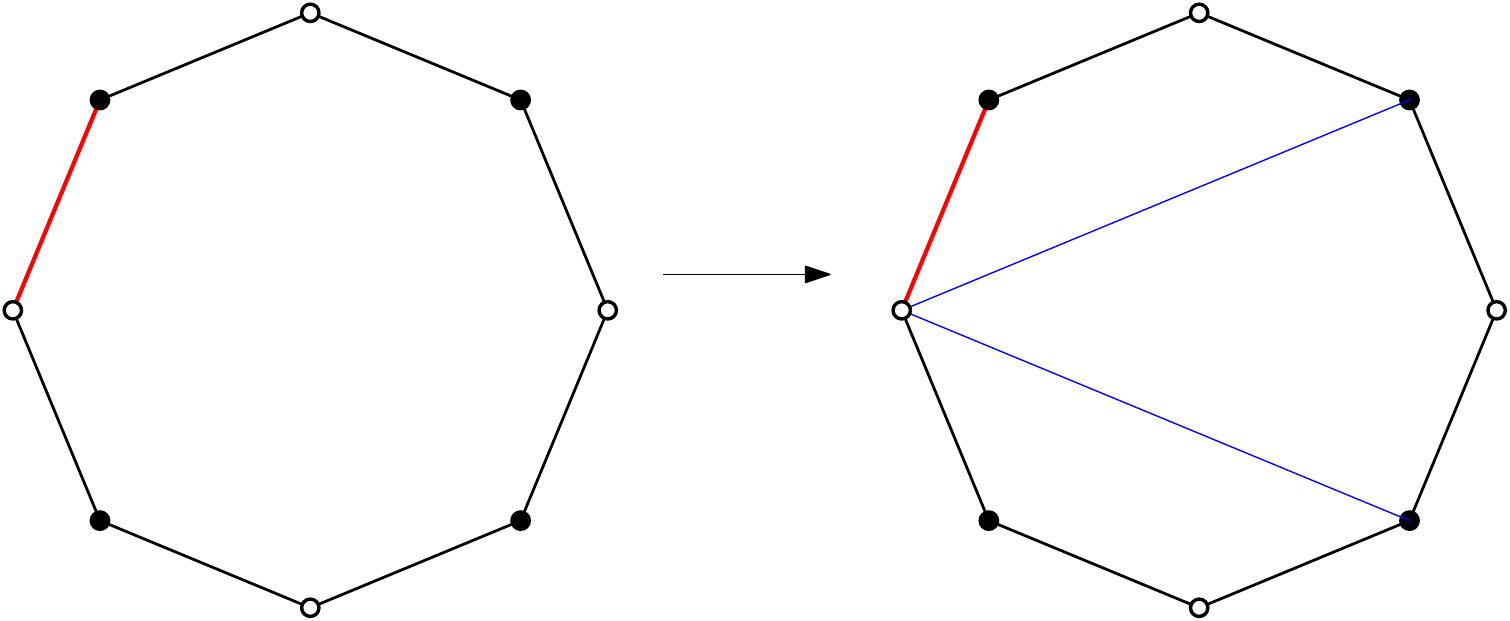}
\caption{Tessellating the boundary. Here $p=4$, the root is in red.}\label{fig_tesselating}
\end{figure}
\end{proof}

The next lemma is in some sense the reverse inequality of Proposition~\ref{prop_injection}.

\begin{lem}\label{lem_ineq_genus}
If $n$ is large enough and $\frac{g}{n}\leq \frac{1}{2
}-\eps$, then the following inequality holds:
\[Q(n,g)\geq C_\eps^2 n^2Q(n,g-1),\] where $C_\eps$ is defined in Lemma~\ref{lem_BRL}.

\end{lem}

\begin{proof}
From $\eqref{eq_CC}$, we directly have
\[Q(n,g)\geq \frac{(2n-2)(n-1)(2n-1)}{n+1}Q(n-2,g-1),\]
and by Lemma~\ref{lem_BRL}, we have $Q(n-2,g-1)\geq C_\eps^2Q(n,g-1)$, and since for $n$ large enough we have $\frac{(2n-2)(n-1)(2n-1)}{n+1}\geq n^2$, the proof is finished.
\end{proof}

Our next goal is to upper bound the sum \[\sum_{\substack{h_1+h_2=g_n\\h_1,h_2\geq 1}} \sum_{n_1+n_2=n}Q(n_1,h_1)Q(n_2,h_2)\] to obtain \eqref{eq_ineq_product}.

Let us introduce the constant $\eps_\theta=\frac{1}{2}\left(\frac{1}{2}-\theta\right)$. The following two lemmas provide estimations of the terms of this sum, in two cases. We start with the case where both $n_1$ and $n_2$ are big enough.

\begin{lem}\label{lem_product_two_big_parts}
If $n$ is large enough, we have the following inequality:
\[\sum_{\substack{h_1+h_2=g_n\\h_1,h_2\geq 1}} \sum_{\substack{n_1+n_2=n\\n_1\geq n_2\geq n^{1/3}}}Q(n_1,h_1)Q(n_2,h_2)\leq \frac{2}{C_{\eps_\theta}^2n^{1/3}}Q(n,g_n).\]
\end{lem}

\begin{proof}
It follows from \eqref{eq_CC} (by forgetting certain terms and constants) and Lemma~\ref{lem_BRL}, that 
\[\sum_{\substack{h_1+h_2=g_n\\h_1,h_2\geq 1}} \sum_{\substack{n_1+n_2=n\\n_1,n_2\geq 0}} n_1Q(n_1,h_1)n_2Q(n_2,h_2)\leq nQ(n+2,g_n)\leq n\frac{1}{\Ce^2}Q(n,g_n),\] and therefore 
\[\sum_{\substack{h_1+h_2=g_n\\h_1,h_2\geq 1}} \sum_{\substack{n_1+n_2=n\\n_1\geq n_2\geq n^{1/3}}} \frac{n}{2}Q(n_1,h_1)n^{1/3}Q(n_2,h_2)\leq n\frac{1}{\Ce^2}Q(n,g_n),\]
which concludes the proof.
\end{proof}

Now we cover the case where $n_2$ is small.

\begin{lem}\label{lem_n2_petit}
For $n$ large enough, we have the following inequality
\[\sum_{\substack{h_1+h_2=g_n\\h_1,h_2\geq 1}} \sum_{\substack{n_1+n_2=n\\1\leq n_2\leq n^{1/3}}}Q(n_1,h_1)Q(n_2,h_2)\leq (1+o(1))\frac{16}{\Ce^4n}Q(n,g_n).\]
\end{lem}

\begin{proof}
Take $h_1,h_2$ such that $h_1+h_2=g_n$ and $h_1,h_2\geq 1$, and $n_1,n_2$ such that $n_1+n_2=n$ and $1\leq n_2\leq n^{1/3}$. 

Note that $\frac{h_1}{n_1}\leq \frac{g_n}{n-n^{1/3}}$, therefore if $n$ is large enough we have $\frac{h_1}{n_1}<\frac{1}{2}-\eps_\theta$. By Lemma~\ref{lem_ineq_genus}, we have $Q(n_1,h_1+1)\geq \Ce^2n_1^2Q(n_1,h_1)\geq \frac{\Ce^2}4 n^2Q(n_1,h_1) $.

By Proposition~\ref{prop_injection}, we have $Q(n_2,h_2-1)\geq \frac{1}{4n_2^3}Q(n_2,h_2)\geq \frac{1}{4n}Q(n_2,h_2)$. Therefore, by an immediate induction, we have 
\[Q(n_1,h_1)Q(n_2,h_2)\leq \left(\frac{16}{\Ce^2n}\right)^{h_2}Q(n_1,g_n)Q(n_2,0).\]
If we sum this inequality over all quadruplets $n_1,n_2,h_1,h_2$, then we obtain 
\[\sum_{\substack{h_1+h_2=g_n\\h_1,h_2\geq 1}} \sum_{\substack{n_1+n_2=n\\1\leq n_2\leq n^{1/3}}}Q(n_1,h_1)Q(n_2,h_2)\leq (1+o(1))\frac{16}{\Ce^2n} \sum_{\substack{n_1+n_2=n\\1\leq n_2\leq n^{1/3}}}Q(n_1,g_n)Q(n_2,0).\]
But
\begin{align*}
\sum_{\substack{n_1+n_2=n\\1\leq n_2\leq n^{1/3}}}Q(n_1,g_n)Q(n_2,0)&\leq \frac{1}{n+3} \sum_{\substack{n_1+n_2=n\\1\leq n_2\leq n^{1/3}}}(2n_1+1)(2n_2+1)Q(n_1,g_n)Q(n_2,0)\\&\leq Q(n+2,g_n)\\&\leq \frac{1}{\Ce^2}Q(n,g_n)
\end{align*} 
where the second inequality holds because of~\eqref{eq_CC}, and the last holds because of Lemma~\ref{lem_BRL}. Therefore, the proof is complete.
\end{proof}

Combining Lemmas~\ref{lem_product_two_big_parts} and~\ref{lem_n2_petit}, one obtains the following inequality (for $n$ large enough):
\begin{equation}
\label{eq_ineq_product}
\sum_{\substack{h_1+h_2=g_n\\h_1,h_2\geq 1}} \sum_{n_1+n_2=n\atop\\n_1,n_2\geq 1}Q(n_1,h_1)Q(n_2,h_2)\leq (1+o(1))\frac{2}{C_{\eps_\theta}^2n^{1/3}}Q(n,g_n).
\end{equation}

\subsection{Proof of Proposition~\ref{prop_cycle_tail}}

There is a bijection between maps with a marked cycle with tail and maps (or pairs of maps) with boundaries.

\begin{lem}\label{lem_couper_cycle_tail}
Let $\cQ_{ct}(n,g,\ell)$ be the set of maps of $\cQ(n,g)$ with a marked cycle with tail of size $\ell$ and $Q_{ct}(n,g,\ell)$ its cardinal. Then, for $n\geq 1, g\geq 1, \ell\geq 1$,
\[Q_{ct}(n,g,\ell)=\sum_{p+p'=\ell\atop p\geq p'\geq 1}(1+\mathbbm{1}_{p\neq p'})\left( Q^{(p,p')}(n,g-1)+\sum_{\substack{n_1+n_2=n\\n_1,n_2\geq 1}} \sum_{\substack{h_1+h_2=g\\h_1,h_2\geq 1}}Q^{(p)}(n_1,h_1)Q^{(p')}(n_2,h_2)\right)\]
\end{lem}

\begin{proof}
Let $(m,(\cP,\C))\in \cQ_{ct}(n,g,\ell)$ where $m\in \cQ(n,g)$ and $(\cP,\C)$ is a cycle with tail of $m$. If $\cP\neq \emptyset$, wlog, say that $\cP$ starts with the root vertex (this will explain the factor $1+\mathbbm{1}_{p\neq p'}$). Let $f$ be the face that lies on the left of the root of $m$. Let $p'=|\C|/2$. Now, cut along $\C$ (see Figure~\ref{fig_coupe_cycle}). There are two possible cases:

\textbf{Case 1:} $\C$ is non separating. Then, after cutting, one obtains a map with two marked simple faces of size $2p'$ that we will call $f_1$ and $f_2$, such that $f_1$ is incident to $\cP$, or adjacent to $f$ if $\cP=\emptyset$. Let $e^*$ be the unique edge incident to $f_1$ and $\cP$ (or to $f_1$ and the root edge if $\cP=\emptyset$) such that $f_1$ lies on the \textbf{left} of $e^*$. We call $e^*$ the gluing edge. Now, on $f_2$, let $e$ be the edge that was identified with $e^*$ before cutting along $\C$, note that $f_2$ lies on the right of $e$, hence we consider it as a root, and now $f_2$ is a boundary. We still need to deal with $f_1$ and $\cP$. If $\cP=\emptyset$, then $f_1$ lies on the right of the root edge, and can be considered as a boundary as well. Otherwise, by cutting along $\cP$ as in Figure~\ref{fig_coupe_path}, we obtain a boundary of size $2p$, with $p=p'+|\cP|$. Notice that the gluing edge $e^*$ is uniquely determined by the position of the root edge, hence we can forget about it.
Now we have a map of $\cQ^{(p,p')}(n,g-1)$ with $p+p'=\ell$ and $p\geq p'\geq 1$. The inverse operation consists in closing the path (such that the root edge gets identified with the other edge of the boundary that is incident to the root vertex) and gluing the two boundaries together by identifying the second root with the gluing edge.


\textbf{Case 2:} $\C$ is  separating. We perform the exact same operation: we cut along $\C$, then cut along $\cP$. We obtain a pair of maps of $\cQ^{(p)}(n_1,h_1)\times \cQ^{(p')}(n_2,h_2)$ with $p+p'=\ell$, $p\geq p'\geq 1$, $n_1+n_2=n$, $n_1,n_2\geq 1$,  $h_1+h_2=g$ and $h_1,h_2\geq 1$. The inverse operation consists in closing the path and gluing the two boundaries together.

\begin{figure}[!h]
\center
\includegraphics[scale=0.6]{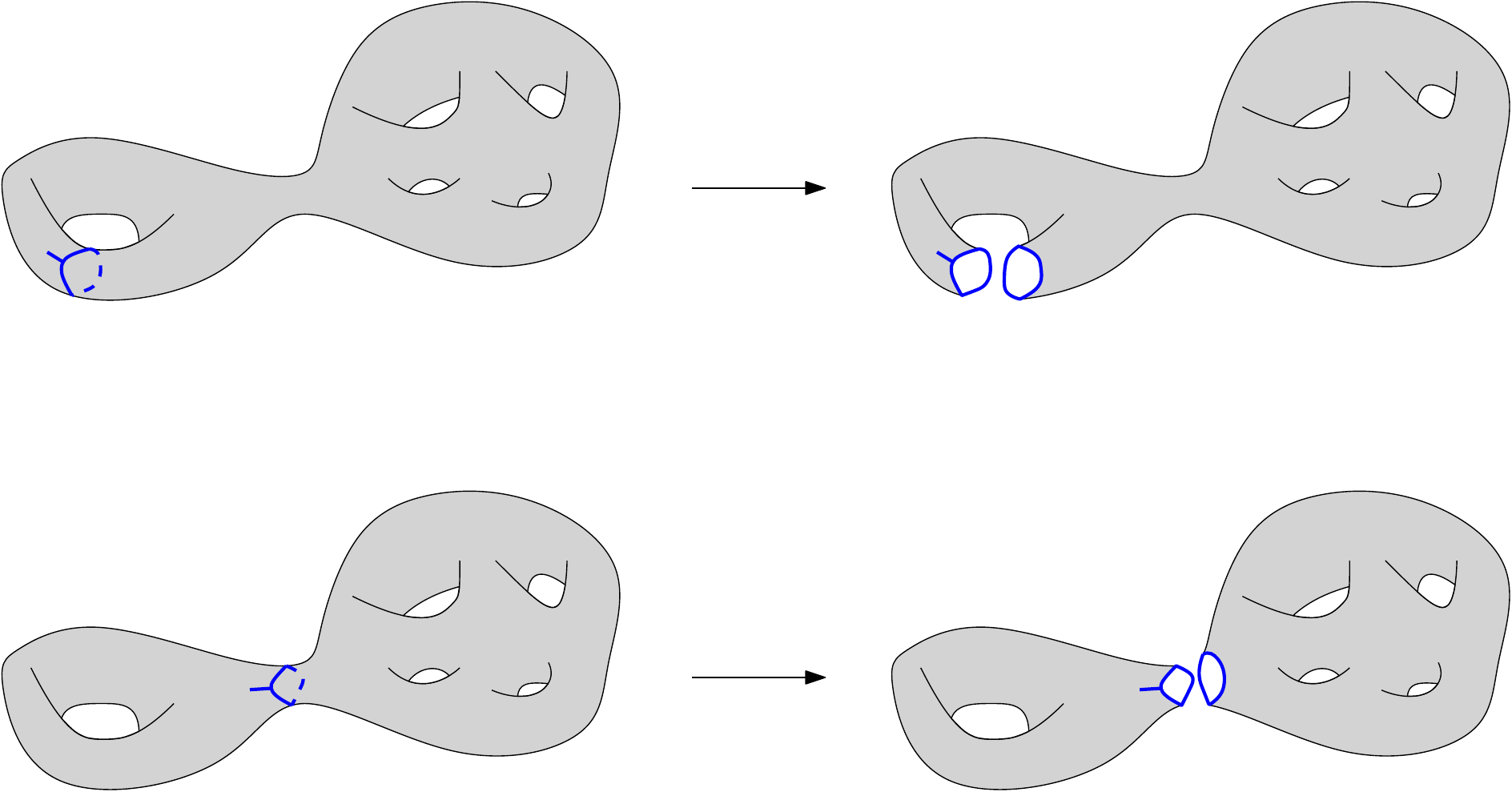}
\caption{Cutting along a non-contractible cycle, the two possible cases}\label{fig_coupe_cycle}
\end{figure}
\begin{figure}[!h]
\center
\includegraphics[scale=0.6]{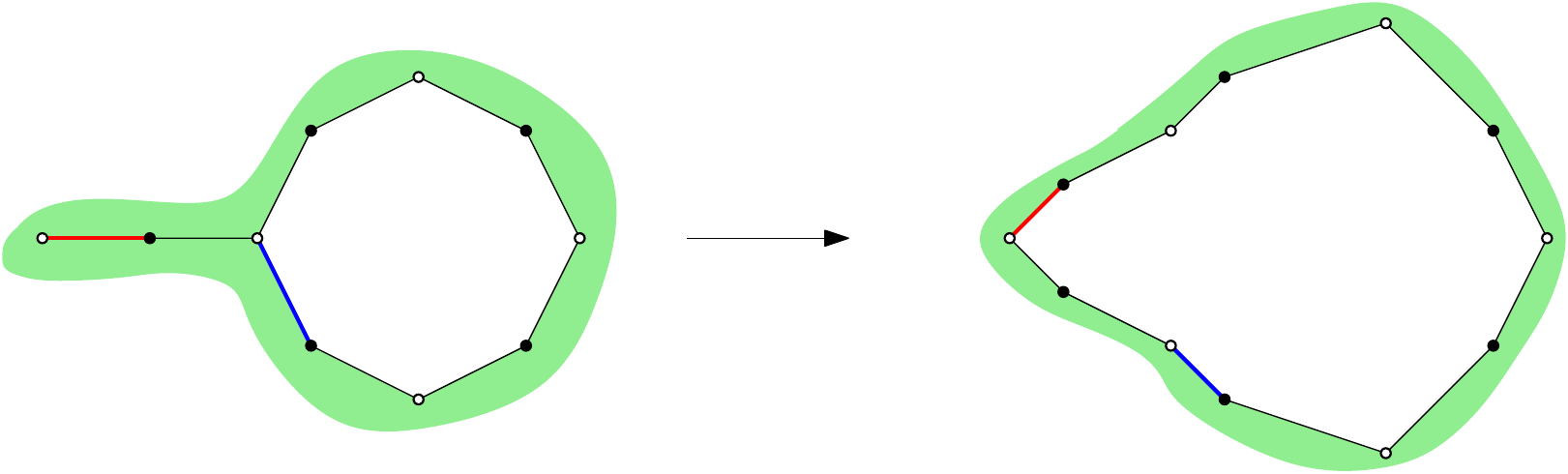}
\caption{Extending a boundary by cutting along a path. Here, $p'=4$ and $p=5$. The root is in red and the gluing edge is in blue.}\label{fig_coupe_path}

\end{figure}
\end{proof}

\begin{proof}[Proof of Proposition~\ref{prop_cycle_tail}]
We will use the first moment method. More precisely, set \[c_\theta=\frac{1}{6\log(1/\Ce)}.\] We will show that for all $\ell\leq c_\theta \log n$, we have $Q_{ct}(n,g_n,\ell)\leq \frac{Q(n,g_n)}{n^{1/7}} $. Then we conclude by a union bound on all $1\leq \ell\leq c_\theta \log n$.

By Lemma~\ref{lem_couper_cycle_tail}, we have (for $n$ large enough):
\begin{align*}
Q_{ct}(n,g_n,\ell)&=\sum_{p+p'=\ell\atop p\geq p'\geq 1}(1+\mathbbm{1}_{p\neq p'})\left( Q^{(p,p')}(n,g_n-1)+\sum_{\substack{n_1+n_2=n\\n_1,n_2\geq 1}} \sum_{\substack{h_1+h_2=g_n\\h_1,h_2\geq 1}} Q^{(p)}(n_1,h_1)Q^{(p')}(n_2,h_2)\right)\\
&\leq 2\ell \left(2(n+\ell-2)Q(n+\ell-2,g_n-1)+\sum_{\substack{n_1+n_2=n+\ell-2\\n_1,n_2\geq 1}} \sum_{\substack{h_1+h_2=g_n\\h_1,h_2\geq 1}}Q(n_1,h_1)Q(n_2,h_2)\right)\\
&\leq 2\ell \left(\frac{2}{\Ce^2 n}Q(n+\ell-2,g_n)+(1+o(1))\frac{2}{C_{\eps_\theta}^2n^{1/3}}Q(n+\ell-2,g_n)\right)\\
&\leq2 c_\theta \log n \left(\frac{1	}{\Ce}\right)^\ell \left(\frac{2}{n}Q(n,g_n)+\frac{2(1+o(1))}{n^{1/3}}Q(n,g_n)\right)\\
&\leq \frac{Q(n,g_n)}{n^{1/7}}
\end{align*}
where in the first inequality we used Lemma~\ref{lem_ineq_boundaries} and the fact that there are less than $\ell$ pairs $(p,p')$ such that $p+p'=\ell$ and $p\geq p'\geq 1$. In the second inequality, we used Lemma~\ref{lem_ineq_genus} for the first term and equation~\eqref{eq_ineq_product} for the second term. Finally, in the third inequality we used Lemma~\ref{lem_BRL}.

\end{proof}
The proof of Proposition~\ref{prop_cycle_tail} finishes the proof of Theorem~\ref{thm_planar_neighborhood}.

\section{Short non-contractible cycles}

Here we will prove Theorem~\ref{thm_short_cycle}, namely that non-separating cycles of length two appear with positive probability. The proof uses the same kind of tools that were used for Theorem~\ref{thm_planar_neighborhood}. This section is quite technical, but the general idea is simple: we estimate the number of maps with one or two marked non-separating cycles of length $2$, and apply the second moment method. Unfortunately, the possibility that the two marked cycles intersect creates a lot of pathological cases that we have to deal with separately, hence the need for many technical lemmas. We will make a heavy use of the Carrell--Chapuy formula~\eqref{eq_CC} as an inequality by forgetting certain terms and constants in order to simplify some expressions as soon as possible. For the sake of simplicity, we will use very rough bounds, hence we will not obtain an optimal value for $k_\theta$ in Theorem~\ref{thm_short_cycle} (we believe that even by being more careful, we cannot find the right value for $k_\theta$).

\subsection{Cutting a cycle of length two}

Here we introduce a bijective operation that we will use a lot in this section: cutting a cycle of length two. It consists in taking a marked cycle in a map, and cutting along it. Depending on whether the cycle was separating or not, we obtain one or two maps as a result, with two disjoint marked digons (one on each map if the original map is cut in two), that we contract into two marked edges. See Figure~\ref{fig_cutting_two} for an illustration. Note that in the separating case, one of the maps inherits the original root, and we root the other one by turning its marked edge into a root. In what follows, when we say we cut a cycle of length two, we mean we apply this precise operation.

\begin{figure}[!h]
\center
\includegraphics[scale=0.4]{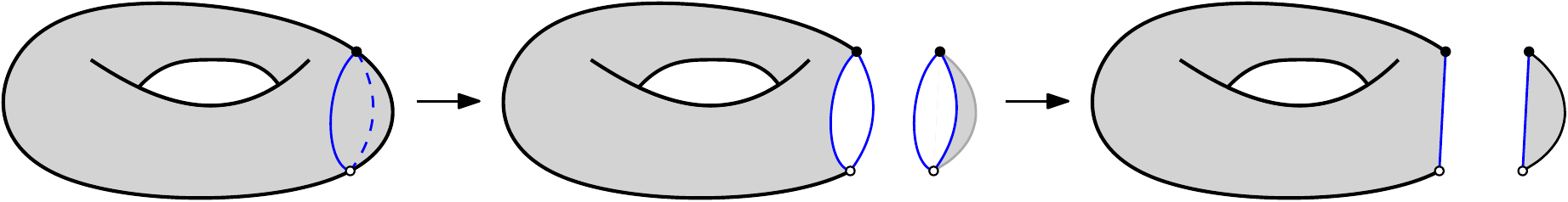}
\includegraphics[scale=0.4]{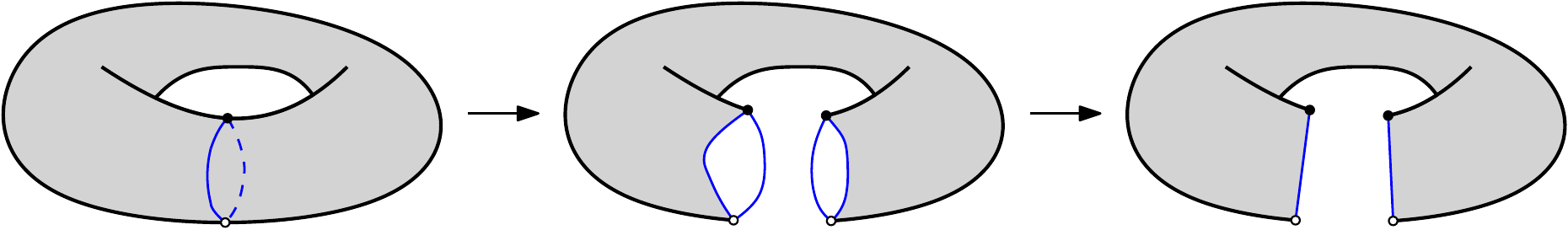}
\caption{Cutting a cycle of length two. Left: the separating case (note that it is possible that both maps are not planar), right: the non-separating case.}\label{fig_cutting_two}
\end{figure}

\subsection{Technical lemmas}
This subsection is devoted to some technical lemmas. The underlying idea is always the same: cutting a cycle and/or filling a face with quadrangles.

\begin{lem}\label{lem_counting_marked_path}
For any $\ell\geq 2$, the number of maps of $\cQ(n,g)$ with a marked path of size $\ell$ is less than $2(n+\ell-1)Q(n+\ell-1,g)$.
\end{lem}

\begin{proof}
We give an injective proof: given such a map, open the path into a face of size $2\ell$ (see Figure~\ref{fig_path_open}), and mark an edge incident to this face to remember how to close it. For instance, say that this marked edge should be glued to the other edge incident to that face that shares a white vertex with it, this uniquely determines how to close the face into a path. Then, tessellate this face with $\ell-1$ quadrangles as in the proof of Lemma~\ref{lem_ineq_boundaries} to obtain a map of $\cQ(n+\ell-1)$ with a marked edge.
\end{proof}

\begin{figure}[!h]
\center
\includegraphics[scale=0.5]{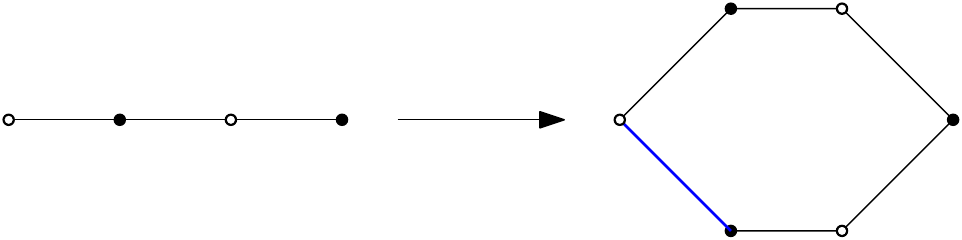}
\caption{Opening a path of length $\ell$ into a face of size $2\ell$. Here, $\ell=3$ and the marked edge is in blue.}\label{fig_path_open}
\end{figure}

\begin{lem}\label{lem_counting_marked_cycle}
For any $\ell\geq 1$, the number of maps of $\cQ(n,g)$ with a marked simple cycle of size $2\ell$ is less than
\[(n+2\ell+1)Q(n+2\ell,g).\]
\end{lem}

\begin{proof}
The proof uses a similar injective operation as in the proof of Lemma~\ref{lem_couper_cycle_tail}, except that now the cycle we consider might be contractible.
Given such a map, cut along this cycle. There are two cases: either it is separating, or not.
If the cycle is not separating, one obtains one map of genus $g-1$ with $n$ quadrangles plus two marked faces of length $2\ell$ with a marked edge on each (to go back, glue the two faces together so that the marked edges coincide). If the cycle is separating, one obtains two maps with $n$ quadrangles and genus $g$ in total, with one marked face on each map, and a marked edge on each face.

Now, as in the proof of Lemma~\ref{lem_ineq_boundaries}, if $\ell=1$, we have marked digons that we close into marked edges. Otherwise, tessellate these marked faces with $\ell-1$ quadrangles each. Note that if we obtain two maps, one of them inherits the original root, and in the other we turn the marked edge into a root. Hence, we have an injective operation into a set of size
\[2(n+2\ell-2)Q(n+2\ell-2,g-1)+\sum_{n_1+n_2=n+2\ell-2\atop\\n_1,n_2\geq 0}\sum_{g_1+g_2=g\atop\\g_1,g_2\geq 0}2n_1Q(n_1,g_1)Q(n_2,g_2).\]
Using~\eqref{eq_CC} we can bound the quantity above by $(n+2\ell+1)Q(n+2\ell,g)$.
\end{proof}

\begin{lem}\label{lem_counting_marked_two cycles}
The number of maps of $\cQ(n,g)$ with two vertex-disjoint marked cycles of length $2$ is less than
\[2n(n+5)Q(n+4,g).\]
\end{lem}
\begin{proof}
Let us write $Q_c(n,g)$ for the number of maps of $\cQ(n,g)$ with \textbf{one} marked cycle of length two. 
Take a map $\cQ(n,g)$ with two vertex-disjoint marked cycles of length $2$, and perform the cutting operation on one of them. Since the two cycles were vertex-disjoint, the second cycle remains a cycle after cutting the first cycle. This is a bijective operation, that puts the set of maps $\cQ(n,g)$ with two disjoint marked cycles of length $2$ in bijection with a set of cardinality
\begin{equation}\label{eq_abc}
n(2n-1)Q_c(n,g-1)+\sum_{n_1+n_2=n\atop\\n_1,n_2\geq 0}\sum_{g_1+g_2=g\atop\\g_1,g_2\geq 0}2n_1(Q_c(n_1,g_1)Q(n_2,g_2)+Q(n_1,g_1)Q_c(n_2,g_2))
\end{equation}
(the argument is exactly the same as in Lemma~\ref{lem_counting_marked_cycle}, except that in case we cut a separating cycle, there are two possibilities, according to where the root and the second cycle are).

Now, using Lemma~\ref{lem_counting_marked_cycle} with $\ell=1$, we upper bound~\eqref{eq_abc} by
\[ 2n(2n-1)(n+3)Q(n+2,g-1)+2n\sum_{n_1+n_2=n\atop\\n_1,n_2\geq 0}\sum_{g_1+g_2=g\atop\\g_1,g_2\geq 0}(n_1+3)Q(n_1+2,g_1)Q(n_2,g_2)+(n_2+3)Q(n_1,g_1)Q(n_2+2,g_2),\]
which, by~\eqref{eq_CC}, is less than
\[2n(n+5)Q(n+4,g)\]
which finishes the proof.
\end{proof}

\subsection{Proof of Theorem~\ref{thm_short_cycle}}

We can now enumerate bipartite quadrangulations with marked non-separating cycles of length $2$.
For $k=1,2$, let $\cQ_{ns}^{(k)}(n,g)$ be the set of bipartite quadrangulations of size $n$, genus $g$ and $k$ marked distinct non separating cycles of length $2$, and $Q_{ns}^{(k)}(n,g)$ be its cardinal.
The cutting operation applied to a non-separating cycle immediately implies
\begin{equation}\label{eq_bij_nonsep}
Q_{ns}^{(1)}(n,g)=n(2n-1)Q(n,g-1).
\end{equation}

Enumerating $Q_{ns}^{(2)}(n,g)$ is a little trickier because the two cycles might intersect. We only give an upper bound.
\begin{lem}\label{lem_bij_nonsep}

We have the following inequality
\begin{align}\label{eq_ineq_two_cycles}
Q_{ns}^{(2)}(n,g)\leq & n(n-1)(2n-1)(2n-3)Q(n,g-2)+2(n+3)Q(n+3,g-1)\notag\\&+2(n+1)(n+5)Q(n+4,g-1)+(4n^2+15n+7)Q(n+2,g-1).
\end{align}
\end{lem}

\begin{proof}

We need to do a careful analysis of all the cases involved when enumerating $Q_{ns}^{(2)}(n,g)$. We only have an inequality because of some pathological cases in which the two cycles intersect. We will perform the same kind of operation, i.e. cutting a cycle into two marked digons, but we need to be cautious, because after cutting the first cycle, the second cycle might not be well defined anymore. All our operations will be injective.

Let us start with a map $m$ of $\cQ_{ns}^{(2)}(n,g)$, we call $c_1$ and $c_2$ its two marked non-separating cycles of length $2$. There are three cases: after we cut $c_1$, $c_2$ might be a cycle, or not. To these two cases, we need to add the "ambiguous" case where $c_1$ and $c_2$ share an edge, which we treat separately. Each of these cases contributes to~\eqref{eq_ineq_two_cycles}. More precisely, the individual contributions can be found in \eqref{eq1}, \eqref{eq2}, \eqref{eq3}, \eqref{eq4}, \eqref{eq5}, \eqref{eq6} and \eqref{eq7}.

\paragraph{Case 1: $c_2$ remains a cycle.}
Let us start with the case where $c_2$ remains a cycle. Recall that we force $c_1$ and $c_2$ to be edge-disjoint, but they do not have to be vertex-disjoint. Let $m'$ be the map obtained after cutting $c_1$. Since $c_1$ and $c_2$ were edge-disjoint in $m$, $c_2$ is disjoint from the two marked edges. Now, we can cut $c_2$. However, in $m'$, it is possible that $c_2$ became a separating cycle, and maybe even contractible (see Figure~\ref{fig_three_cases}).

\begin{figure}[!h]
\center
\includegraphics[scale=0.7]{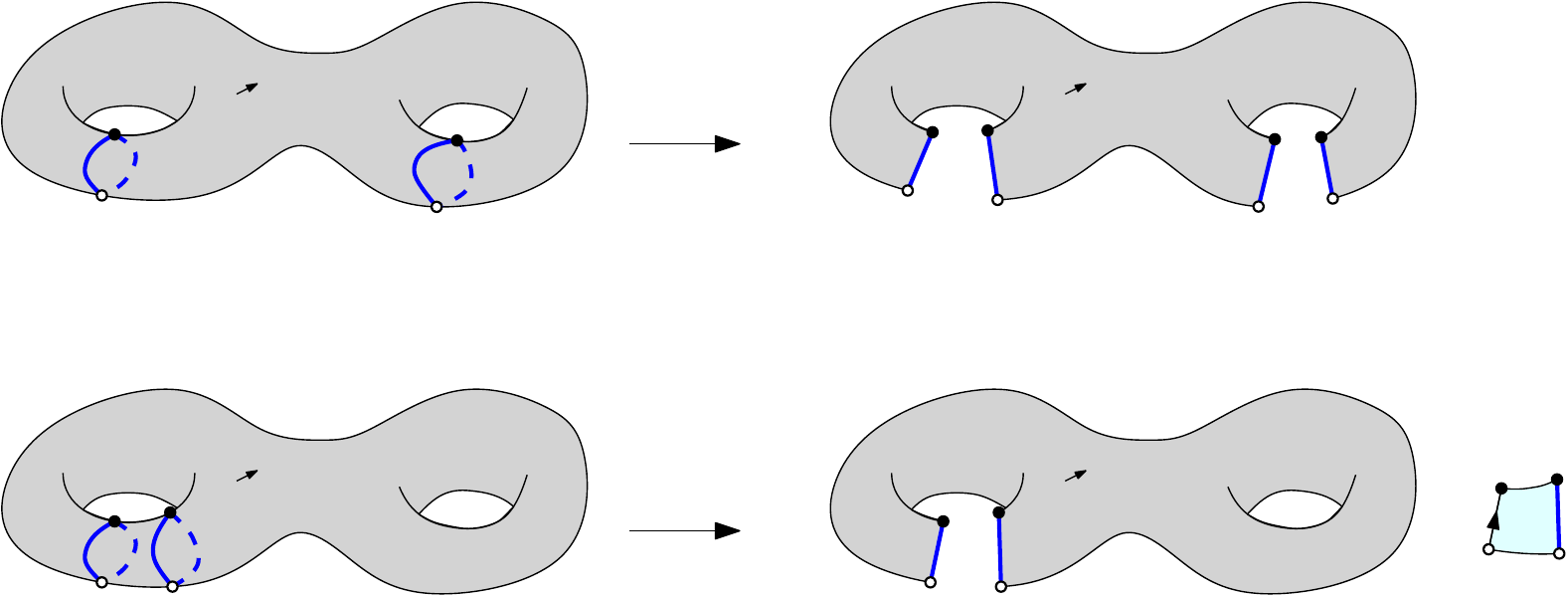}
\caption{Cutting two short cycles. Above: when the second cycle stays non-separating, below: when it becomes separating (note that the cyan map doesn't have to be planar).}\label{fig_three_cases}
\end{figure}
If $c_2$ is non-separating in $m'$, we obtain a map of genus $g-2$ with $n$ quadrangles and two distinct unordered pairs of marked edges, which contributes to exactly 
\begin{equation}\label{eq1}
n(n-1)(2n-1)(2n-3)Q(n,g-2)
\end{equation}
in~\eqref{eq_ineq_two_cycles}. 
 
If $c_2$ was separating, then we obtain two maps of total genus $g-1$ and total number of quadrangles $n$, with two distinguishable marked edges on each maps (they are distinguishable because on each map, one of these edges comes from $c_1$ and the other comes from $c_2$). One of these two maps contains the root of $m$, in the second one, we turn one of its marked edges into a root. Therefore, the contribution of this case in~\eqref{eq_ineq_two_cycles} is 
\[\sum_{n_1+n_2=n\atop n_1,n_2\geq 1} \sum_{g_1+g_2=g-1\atop g_1,g_2\geq 0}2n_1(2n_1-1)Q(n_1,g_2)2n_2Q(n_2,g_2)\]
If we upper bound $n_1$ by $n$ in the sum above, by~\eqref{eq_CC}, we obtain that this contribution is less than
\begin{equation}\label{eq2}
2n(n+3)Q(n+2,g-1).
\end{equation}


\paragraph{Case 2a: $c_2$ is not a cycle anymore, but stays connected.}

In this case, $c_1$ and $c_2$ were not vertex-disjoint. When we cut $c_1$, we either obtain a map of $\cQ(n,g-1)$ with a marked path of length $4$ (if $c_1$ and $c_2$ shared only vertex, and the extremities of the path are of the color of the vertex $c_1$ and $c_2$ did not share) or a marked simple cycle of length $4$ (if $c_1$ and $c_2$ shared two vertices). See Figure~\ref{fig_coupe_cycle_ou_path} for an illustration.
By Lemma~\ref{lem_counting_marked_path} (for $\ell=4$) and Lemma~\ref{lem_counting_marked_cycle} (for $\ell=2$), this contributes to at most 
\begin{equation}\label{eq3}
2(n+3)Q(n+3,g-1)+2(n+5)Q(n+4,g-1)
\end{equation}
in \eqref{eq_ineq_two_cycles}.
The factor $2$ in the second term comes from the fact that,  in the cycle of length $4$, we need to remember which of the edges belonged to $c_1$ and which belonged to $c_2$ (two possibilities).

\begin{figure}[!h]
\center
\includegraphics[scale=0.5]{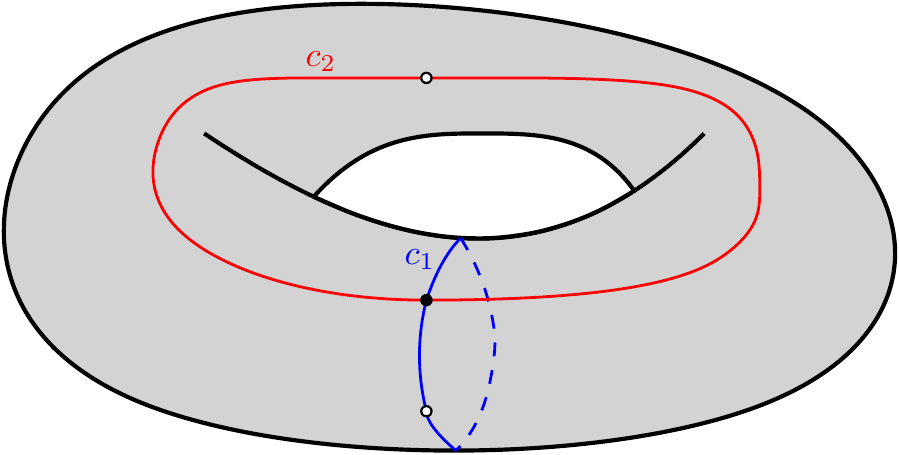}
\includegraphics[scale=0.5]{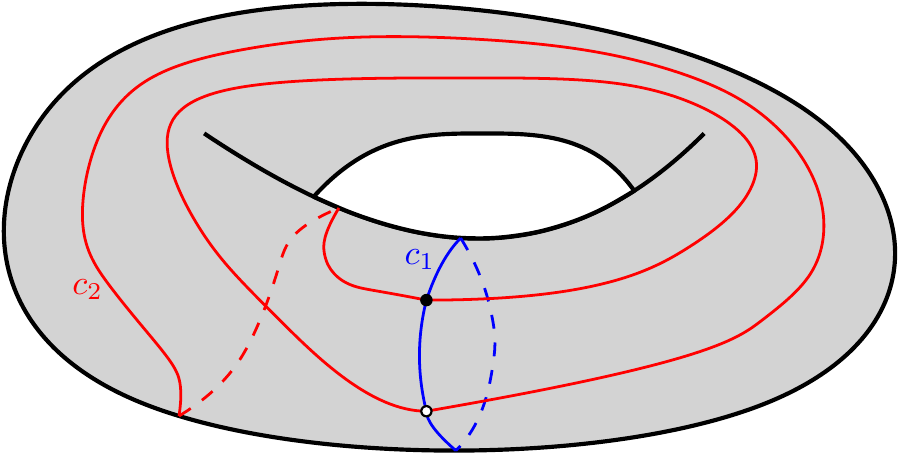}
\caption{The case where $c_2$ is not a cycle anymore but remains connected when $c_1$ is cut. Left: when $c_1$ and $c_2$ share one vertex. Right: when they share two vertices.}\label{fig_coupe_cycle_ou_path}
\end{figure}

%
%
%

\paragraph{Case 2b: $c_2$ is not a cycle anymore, and doesn't stay connected.}

In this case, $c_1$ and $c_2$ have to share two vertices.
Cut $c_1$ to obtain two marked edges $a_1$ and $b_1$. Call $a_2$ and $b_2$ the two edges of $c_2$. Say that, wlog, $a_1$ and $a_2$ (resp. $b_1$ and $b_2$) are incident. What happens (see Figure~\ref{fig_cycles_disconnect} for an example) then is we obtain a map of $\cQ(n,g-1)$ with two \textbf{disjoint} marked cycles of length $2$ ($(a_1,a_2)$ and $(b_1,b_2)$), which by Lemma~\ref{lem_counting_marked_two cycles} contributes to less than
\begin{equation}\label{eq4}
2n(n+5)Q(n+4,g-1)
\end{equation}
in~\eqref{eq_ineq_two_cycles}. 

\begin{figure}[!h]
\center
\includegraphics[scale=0.7]{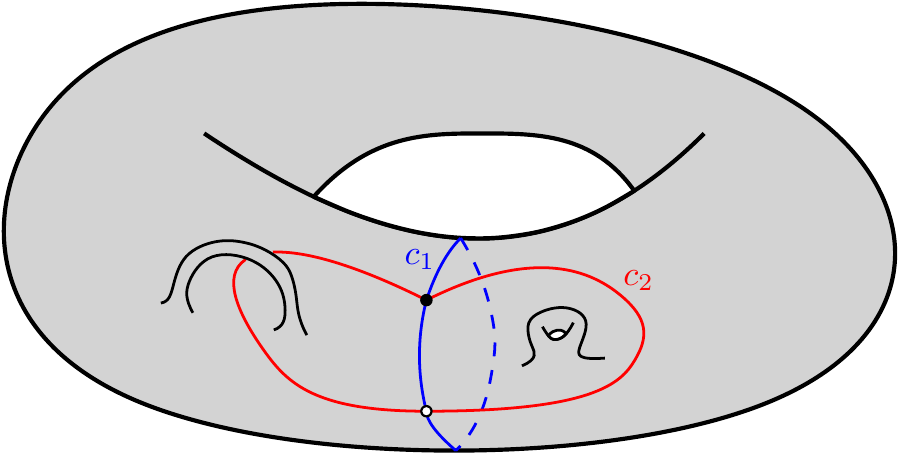}
\caption{An example of the case where $c_2$ gets disconnected when $c_1$ is cut.}\label{fig_cycles_disconnect}
\end{figure}

\paragraph{Case 3: $c_1$ and $c_2$ share an edge.}

Let us say that $c_1=(e_1,e)$ and $c_2=(e_2,e)$. There are two cases.

\textbf{a)} $(e_1,e_2)$ is a separating cycle. We cut it and obtain two maps $m_1$ with a marked edge and $m_2$ with a marked edge on a non-separating cycle of length $2$ (see Figure~\ref{fig_shared_sep}). Again, one of the two maps doesn't have a root, hence we need to turn one of the marked edges into a root edge.

If it is $m_2$ that still has a marked edge, then there are exactly
\[\sum_{n_1+n_2=n\atop n_1,n_2\geq 1} \sum_{g_1+g_2=g\atop g_1\geq 0,g_2\geq 1} Q(n_1,g_1)2Q_{ns}^{(1)}(n_2,g_2)=\sum_{n_1+n_2=n\atop n_1,n_2\geq 1} \sum_{g_1+g_2=g\atop g_1\geq 0,g_2\geq 1} Q(n_1,g_1)2n_2(2n_2-1)Q(n_2,g_2-1) \]
cases, where the equality comes from~\eqref{eq_bij_nonsep}.
We can upper bound $n_2$ by $n$ and use~\eqref{eq_CC} to bound the expression above by
\begin{equation}\label{eq5}
2n(n+3)Q(n+2,g-1).
\end{equation}

Otherwise, the root of $m_2$ lies on a non-separating cycle of length $2$ and $m_1$ has a marked edge. By a similar reasoning as in~\eqref{eq_bij_nonsep}, the number of cases is

\[\sum_{n_1+n_2=n\atop n_1,n_2\geq 1} \sum_{g_1+g_2=g\atop g_1\geq 0,g_2\geq 1} 2n_1Q(n_1,g_1)2(2n_2-1)Q(n_2,g_2-1) \]
which is less than
\begin{equation}\label{eq6}
(n+3)Q(n+2,g-1)
\end{equation}
by~\eqref{eq_CC}.

\begin{figure}[!h]
\center
\includegraphics[scale=0.7]{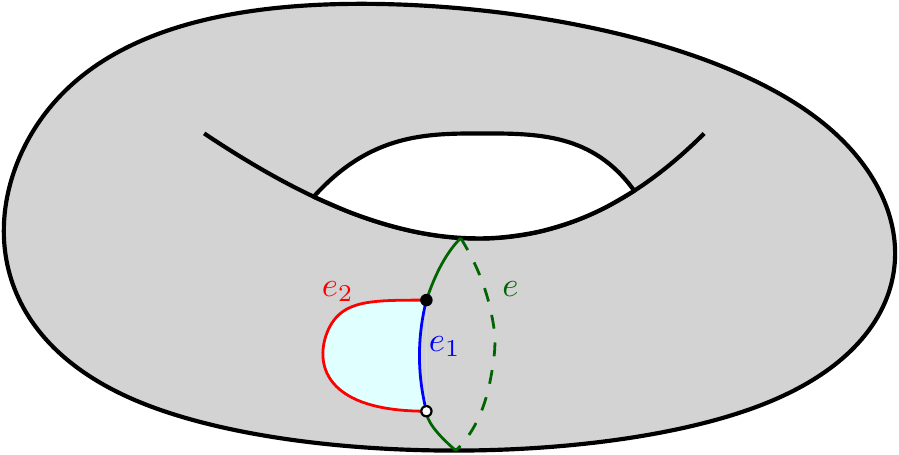}
\caption{When $c_1$ and $c_2$ share an edge and $(e_1,e_2)$ is a separating cycle. Note that the cyan part doesn't have to be planar.}\label{fig_shared_sep}
\end{figure}

\textbf{b)} Finally, the only remaining case is when $(e_1,e_2)$ is non-separating. Then we cut it, and by a similar reasoning as in Case 2a, we obtain a map of genus $g-1$ with a marked path of length $3$, which by Lemma~\ref{lem_counting_marked_path} contributes to less than
\begin{equation}\label{eq7}
2(n+2)Q(n+2,g-1)
\end{equation}
in \eqref{eq_ineq_two_cycles}.

\begin{figure}[!h]
\center
\includegraphics[scale=0.7]{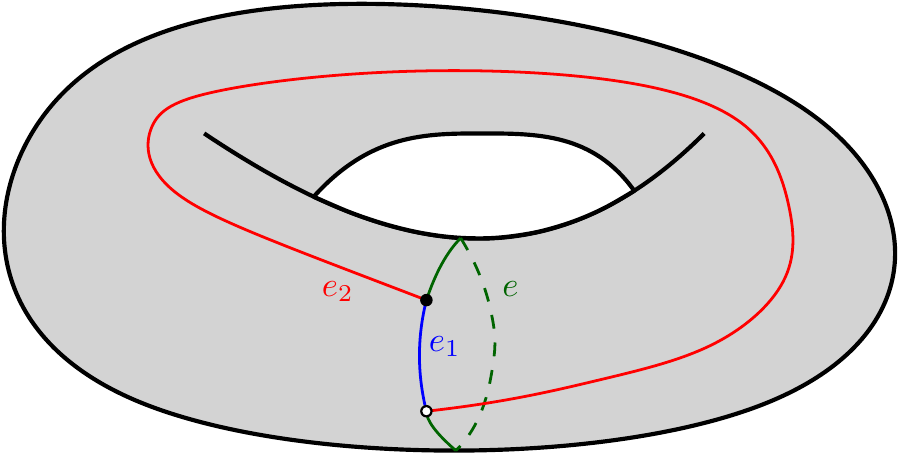}
\caption{When $c_1$ and $c_2$ share an edge and $(e_1,e_2)$ is non-separating.}
\end{figure}




\end{proof}

We are now ready to prove Theorem~\ref{thm_short_cycle}.

\begin{proof}[Proof of Theorem \ref{thm_short_cycle}]
Let $X_{n,g}$ be the number of non-separating cycles of length $2$ in a uniform map  $m\in\cQ(n,g)$, then $X_{n,g}^2$ is the number of ordered pairs of distinct non-separating cycles of length $2$, plus $X_{n,g}$.
Hence
\[\E(X_{n,g})=\frac{Q_{ns}^{(1)}(n,g)}{Q(n,g)}\]
and
\[\E(X_{n,g}^2)=\E(X_{n,g})+\frac{Q_{ns}^{(2)}(n,g)}{Q(n,g)}.\]
Now, recall that $g_n$ is a sequence such that $\frac{g_n}{n}\to \theta \in (0,1/2)$. We want to prove 
\[\P(X_{n,g_n}>0)\geq k_\theta +o(1)\]
whp.
We will use the second moment method, namely the fact that 
\begin{equation}\label{eq_second_moment}
P(X_{n,g_n}>0)\geq \frac{\E(X_{n,g_n})^2}{\E(X_{n,g_n}^2)}.
\end{equation}

By \eqref{eq_bij_nonsep} and Proposition~\ref{prop_injection}, we have
\[Q_{ns}^{(1)}(n,g_n)=(1+o(1))2n^2Q(n,g_n-1)\geq \left(\frac{\theta}{2}+o(1)\right)Q(n,g_n).\]
Therefore
\begin{equation}\label{eq_low_bound_first_moment}
\E(X_{n,g_n})\geq \frac{\theta}{2}+o(1).
\end{equation}

Now we will upper bound $Q_{ns}^{(2)}(n,g_n)$.
First, note by Lemmas~\ref{lem_BRL} and~\ref{lem_ineq_genus} that 
\[2(n+3)Q(n+3,g_n-1)=o(Q(n,g_n)),\]
hence
\[Q_{ns}^{(2)}(n,g_n)\leq (1+o(1))\left(4n^4Q(n,g_n-2)+4n^2Q(n+2,g_n-1)+2n^2Q(n+4,g_n-1)\right).\]
Applying Lemmas~\ref{lem_BRL} and~\ref{lem_ineq_genus} to the inequality above, we obtain 
\[Q_{ns}^{(2)}(n,g_n)\leq(1+o(1))\left( \frac{4}{\Ce^4}+\frac{4}{\Ce^4}+\frac{2}{\Ce^6}\right)Q(n,g_n),\]
hence
\begin{equation}\label{eq_up_bound_second_moment}
\E(X_{n,g_n}^2)\leq \E(X_{n,g_n})+\frac{8}{\Ce^4}+\frac{2}{\Ce^6}+o(1).
\end{equation}

Now, by \eqref{eq_second_moment}, Theorem~\ref{thm_short_cycle} holds for 
\[k_\theta=\left(\frac{2}{\theta}+\frac{32}{\Ce^4 \theta^2}+\frac{8}{\Ce^6 \theta^2} \right)^{-1}.\]
\end{proof}

\bibliographystyle{abbrv}
\bibliography{bibli}

\end{document}